		\documentclass[10pt]{elsarticle}
		\usepackage{etoolbox}
		\makeatletter
		\patchcmd{\ps@pprintTitle}{\footnotesize\itshape
		Preprint submitted to \ifx\@journal\@empty Elsevier
		\else\@journal\fi\hfill\today}{\relax}{}{}
		\makeatother
\usepackage{lineno,hyperref}
\modulolinenumbers[5]
\usepackage{amsfonts}
\usepackage{amssymb}
\usepackage{amsmath}
\usepackage{amsthm}
\usepackage{graphicx}
\usepackage{stackrel}
\usepackage{filecontents}
\usepackage{enumerate}
\newtheorem{theorem}{Theorem}[section]
\newtheorem{lemma}{Lemma}[section]

\journal{arXiv}
\def\tr{{\rm{tr}}}
\usepackage{xcolor}
\usepackage{epstopdf}
\usepackage{enumitem}

\begin{document}

\begin{frontmatter}
\title{A Holling--Tanner predator-prey model with strong Allee effect}
\author[QUTaddress,UDLAaddress]{Claudio Arancibia--Ibarra}
\author[USDaddress]{Jos\'e Flores}
\author[QUTaddress]{Graeme Pettet}
\author[QUTaddress]{Peter van Heijster}
\address[QUTaddress]{School of Mathematical Sciences, Queensland University of Technology, \\ 
GPO Box 2434, GP Campus, Brisbane, Queensland 4001 Australia\\
claudio.arancibia@hdr.qut.edu.au}
\address[UDLAaddress]{Facultad de Educaci\'on, Universidad de Las Am\'ericas,\\
Av. Manuel Montt 948, Santiago, Chile}
\address[USDaddress]{Department of Computer Science, The University of South Dakota\\
Vermillion, SD 57069, South Dakota, USA}
\begin{abstract}
We analyse a modified Holling--Tanner predator-prey model where the predation functional response is of Holling type II and we incorporate a strong Allee effect associated with the prey species production. The analysis complements results of previous articles by Saez and Gonzalez-Olivares (SIAM J. Appl. Math.~59 1867--1878, 1999) and Arancibia-Ibarra and Gonzalez-Olivares (Proc. CMMSE 2015 130--141, 2015)
discussing Holling--Tanner models which incorporate a weak Allee effect. The extended model exhibits rich dynamics and we prove the existence of separatrices in the phase plane separating basins of attraction related to co-existence and extinction of the species. We also show the existence of a homoclinic curve that degenerates to form a limit cycle and discuss numerous potential bifurcations such as saddle-node, Hopf, and Bogadonov--Takens bifurcations.
\end{abstract}
\end{frontmatter}

\newpage
\section{Introduction}
Predator-prey models have been studied extensively in the field of mathematical ecology to describe the dynamic interactions and variations in populations over time. 
In particular the Holling--Tanner models have been shown to be very effective in describing real predator-prey systems such as those describing mite and spider-mite \cite{kozlova}, Canadian lynx and snowshoe hare \cite{chivers}, and sparrow and sparrowhawk \cite{moller} interactions. For instance, Wolkind {\em et al.\ }\cite{wollkind} use a particular Lotka--Volterra model to investigate the dynamics of a pest in fruit-bearing trees, under the hypothesis that the parameters depend on the temperature.
The dynamic complexities in these Holling--Tanner models are also of mathematical interest, both on  temporal domains \cite{arancibia2, arrows, saez, yu, zhao} and on spatio-temporal domains \cite{banerjee1, ghazaryan}.  

The following pair of equations is a typical representation of a purely temporal Holling--Tanner model, where $x(t)$ is used to represent the size of the prey population at time $t$, and $y(t)$ is used to represent the size of the predator population at time $t$;
\begin{equation}\label{prey-predator}
\begin{aligned}
\dfrac{dx}{dt} &=	 rx\left( 1-\dfrac{x}{K}\right) - H(x)y\,, \\
\dfrac{dy}{dt} &=	 sy\left( 1-\dfrac{y}{nx}\right)\,.
\end{aligned}
\end{equation} 
Here, $r$ and $s$ are the intrinsic growth rate for the prey and predator respectively, $n$ is a measure of the quality of the prey as food for the predator, $K$ is the prey environmental carrying capacity, and $H(x)$ is the predation functional response used to represent the impact of predation upon the prey species. The growth of the predator and prey population is of logistic form and all the parameters are assumed to be positive. The behaviour of the Holling--Tanner model depends on the type of the predation functional response chosen. Holling proposed three principal type of functional response\cite{holling} representing different behaviors of the types of species: Type I is a linear increasing function, Type II is hyperbolic in form, and  Type III is a sigmoid. In this manuscript, we will use a Holling Type II functional response. This type of functional response occurs in species when the number of prey consumed rises rapidly at the same time that the prey densities increases\cite{may,turchin}. In (\ref{prey-predator}) the Type II functional response adopted corresponds to $H(x) = qx/(x+a)$, where $q$ is the maximum predation rate per capita and $a$ is half of the saturated response level \cite{turchin}.
The resulting Holling--Tanner model is an autonomous two-dimensional system of differential equations and is given by
\begin{equation}\label{ht1}
\begin{aligned}
\dfrac{dx}{dt} & = rx \left( 1-\dfrac{x}{K}\right)-\dfrac{qxy}{x+a}\,, \\ 
\dfrac{dy}{dt} & = sy\left( 1\ -\dfrac{y}{nx}\right) \,.
\end{aligned}  
\end{equation}
This system is singular when $x(t)=0$, but it can be desingularised to give a topologically equivalent system as studied by Saez and Gonzalez--Olivares \cite{saez}. The authors analysed the global stability of the unique equilibrium point in the first quadrant of this equivalent system and showed that  the equilibrium point can be stable, unstable surrounded by a stable limit cycle or stable surrounded by two limit cycles.

An Allee effect is a density-dependent phenomenon in which fitness, or population growth, increases as population density increases \cite{allee, berec, courchamp,kramer, stephens}. This phenomenon is also called depensation in fisheries sciences and positive density dependence in population dynamics \cite{liermann}.	
In ecology, these mechanisms are connected with individual cooperation such as strategies to hunt, collaboration in unfavourable abiotic conditions, and reproduction \cite{courchamp3}. When the population density is low species might have more resources and benefits. However, there are species that may suffer from a lack of conspecifics. This may impact their reproduction or reduce the probability to survive when the population volume is low \cite{courchamp2}. In general, the Allee effect may appear due to a wide range of biological phenomena, such as reduced anti-predator vigilance, social thermo-regulation, genetic drift, mating difficulty, reduced defense against the predator, and deficient feeding because of low population densities \cite{stephens2}.

With an Allee effect included, the Holling--Tanner Type II model (\ref{ht1}) becomes
\begin{equation}
\begin{aligned} \label{hta1}
\dfrac{dx}{dt} &= rx\left(1-\dfrac{x}{K}\right)(x-m) - \dfrac{qxy}{x+a}\,,\\
\dfrac{dy}{dt} &= sy\left(1-\dfrac{y}{nx} \right).
\end{aligned}
\end{equation}
The growth function $\ell(x) = rx(1-x/K)(x-m)$ has an enhanced growth rate as the population increases above the threshold population value $m$. If $\ell(0)=0$ and $\ell'(0) \geq 0$ - as it is the case with $m  \leq 0$ - then $\ell(x)$ represents a proliferation exhibiting a weak Allee effect, whereas if $\ell(0)=0$ and $\ell'(0)<0$ - as it is the case with $m>0$ - then $\ell(x)$ represents a proliferation exhibiting a strong Allee effect \cite{yue2}. 

The aim of this manuscript is to study the dynamics of the Holling--Tanner predator-prey model with strong Allee effect on prey and functional response Holling Type II, that is \eqref{hta1} with $m>0$. The main difference between system \eqref{ht1} and \eqref{hta1} is the fact that \eqref{hta1} has at most two equilibrium points in the first quadrant instead of one for \eqref{ht1}. This additional equilibrium point gives rise to saddle-node bifurcations and Bogadonov--Takens bifurcations, and significantly enriches the dynamics of the system. This manuscript also extends some of the results obtained by Arancibia--Ibarra and Gonz\'alez--Olivares \cite{arancibia} for a modified Holling--Tanner model with $m=0$, that is, with a specific weak Allee effect.  The authors showed the existence of a stable limit cycle, so that both species vary periodically. The model with $m<0$ has not been studied in detail yet since the dynamics and analysis is more complicated as there are (at most) three equilibrium points in the first quadrant. In addition, it complements the results of the Holling--Tanner model \eqref{ht1} studied by Saez and Gonzalez--Olivares \cite{saez}. 

A Holling-Tanner model considering Allee effects and functional response Holling Type I has been studied by Gonz\'alez--Olivares {\em et al.\ } \cite{gonzalez6}. The authors showed that there exist a region in parameter space where the model has a stable limit cycle. A Holling-Tanner model considering a weak Allee effect and functional response Holling Type III has been studied by Tintinago--Ruz {\em et al.\ } \cite{tintinago}. The authors showed that the system, for certain system parameters, has one equilibrium point in the first quadrant. This equilibrium point can be unstable surrounded by a stable limit cycle. Pal {\em et al.\ } \cite{pal} incorporated a Beddington--DeAngelis functional response where the predation rate is  both predator and prey dependant and the prey grow rate is affected by a strong Allee effect.  In this article the authors showed a Hopf bifurcation in presence of delay and the stability of a limit cycle. Moreover, the Holling Type III and Beddington--DeAngelis functional response can not be reduced to Holling Type II response function as the analysis in both papers is based on the sigmoid form of the response function and, consequently, some of the system parameters are necessary positive in both papers. Therefore, the dynamics and analysis are significantly different. For instance, the Holling--Tanner model with Beddington--DeAngelis functional response and strong Allee effect only supports stable limit cycles and the origin is not singular \cite{pal}. In contrast, for \eqref{hta1} the origin is singular and there exists a system parameters where \eqref{hta1} has an unstable limit cycle, see for instance Figure \ref{fig:06} in Subsection \ref{nateq} in where we show this behaviour.  

The Holling--Tanner model with strong Allee effect is discussed further in Section~\ref{SS:MODEL} and a topological equivalent model is derived. In Section~\ref{SS:RES}, we study the main properties of the model. That is, we prove the stability of the equilibrium points and give the conditions for saddle-node bifurcations and Bogadonov--Takens bifurcations. We conclude the manuscript summarising the results and discussing the ecological implications.

\section{The Model}
\label{SS:MODEL}
The Holling--Tanner model with strong Allee effect is given by \eqref{hta1} with $m>0$, and for biological reasons we only consider the model in the domain $\Omega=\{(x,y)\in\mathbb{R}^2,x>0, y\geq0\}$ and $a<K$. The equilibrium points of system \eqref{hta1} are $(K,0)$, $(m,0)$, and $(x^*,y^*)$, this last point(s) being defined by the intersection of the nullclines $y=nx$ and $y=r(1-x/K)(x-m)(x+a)/q$.  In order to simplify the analysis, we follow \cite{blows, gonzalezy2,saez} and convert \eqref{hta1} to a topologically equivalent nondimensionalised model that has fewer parameters and that is no longer singular along $x=0$, see also \cite{harley,pettet}. 
\begin{theorem}\label{the1}
System \eqref{hta1} is topologically equivalent to system
\begin{equation}\label{hta2} 
\begin{aligned}
\dfrac{du}{d\tau} & =  \ u^2\left(\left(u+A\right)\left( 1-u\right)\left(u-M \right) -Qv\right) \,,\\
\dfrac{dv}{d\tau} & =  \ S\left(u+A\right)\left(u-v\right)v \,,
\end{aligned}
\end{equation}
in $\Omega$.
\end{theorem}
\begin{proof}
We introduce a change of variable and time rescaling, given by the function $\varphi :\breve{\Omega}\times\mathbb{R}\rightarrow \Omega\times\mathbb{R}$, where $\varphi(u,v,\tau)=(x,y,t)$ is defined by $x=Ku$, $y=nKv$, $d\tau =rK\,dt/(u(u+a/K))$ and $\breve{\Omega}=\{(u,v)\in \mathbb{R}^2, u>0, v\geq0\}$. Replacing $x=Ku$ and $y=nKv$ in \eqref{hta1}, we obtain
\[\begin{aligned}
\dfrac{du}{dt} & = rK\left(\left(1-u\right)\left(u-\dfrac{m}{K}\right)-\dfrac{nq }{rK\left(u+\dfrac{a}{K}\right) }v\right)u\,, \\ 
\dfrac{dv}{dt} & = s\left(1-\dfrac{v}{u}\right)v\,.
\end{aligned}\]
Next, we rescale time in such a way as to eliminate the singularities. That is, we rescale $rK\,dt = u(u+a/K)\,d\tau$ to obtain
\[\begin{aligned}
\dfrac{du}{d\tau} & = u^2\left(\left(u+\dfrac{a}{K}\right)\left( 1-u\right)\left(u-\dfrac{m}{K} \right) -\dfrac{nq}{rK}v\right)\,,\\
\dfrac{dv}{d\tau} & = \dfrac{s}{rK}\left(u+\dfrac{a}{K}\right)\left(u-v\right)v\,.
\end{aligned}\]
System (\ref{hta2}) is obtained upon defining $A:=a/K\in(0,1)$, $S:=s/(rK)$, $Q:=nq/(rK)$ and $M:=m/K\in(0,1)$, so $(A,M,S,Q)\in\Pi=(0,1)\times(0,1)\times \mathbb{R}^2_+$. In addition, the function $\varphi$ is a diffeomorphism preserving the orientation of time since $\det (\varphi(u,v,\tau))=nu(a+Ku)/r>0$ \cite{chicone}, see Figure \ref{dif}.
\end{proof}
\begin{figure}
\centering
\includegraphics[width=12cm]{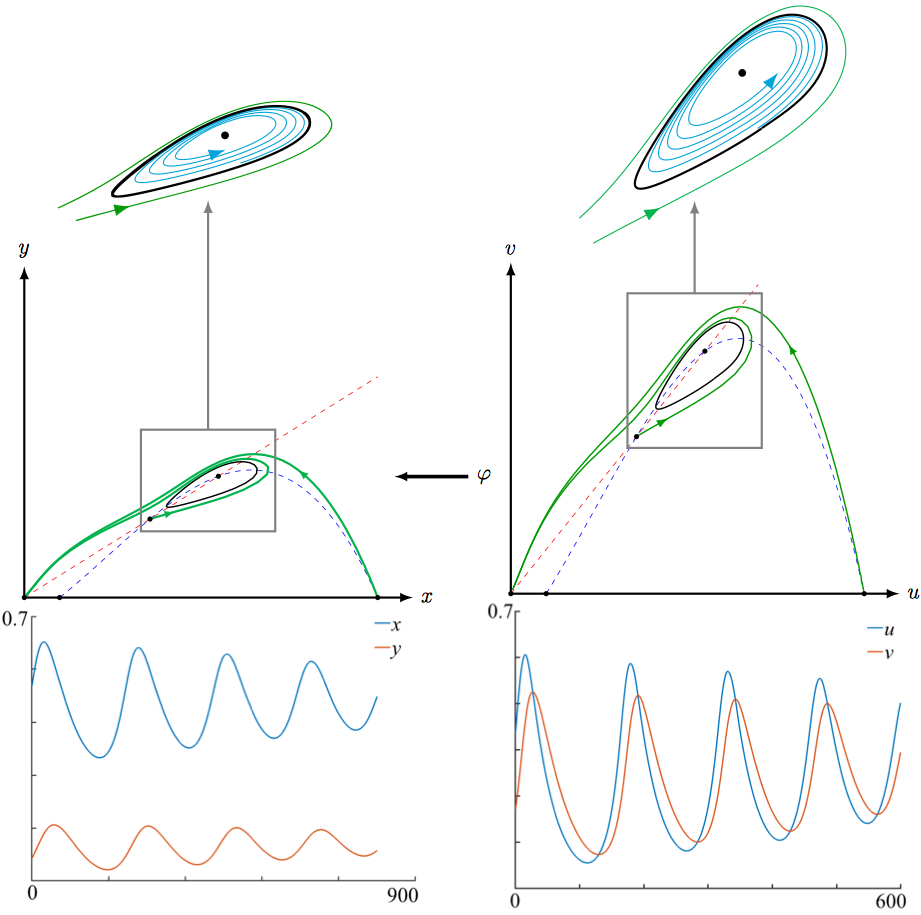}
\caption{The panels on the left represent the dynamics of system \eqref{hta1}. The panels on the right illustrate the topologically equivalent system \eqref{hta2} obtained via the diffeomorphism $\varphi$.}
\label{dif}
\end{figure}
So, instead of analysing system \eqref{hta1} we analyse the topologically equivalent, system \eqref{hta2}.
As $du/d\tau=uR(u,v)$ and $dv/d\tau=vW(u,v)$ with $R(u,v)=u(u+A)( 1-u)(u-M)-Quv$ and $W(u,v)=S(u+A)(u-v)$ system \eqref{hta2} is of Kolmogorov type. That is, the axes $u=0$ and $v=0$ are invariant. The $u$-nullclines of system \eqref{hta2} are $u=0$ and $v=(u+A)(1-u)(u-M)/Q$, while the $v$-nullclines are $v=0$ and $v=u$. Hence, the equilibrium points for system \eqref{hta2} are  $(0,0)$, $(1,0)$, $(M,0)$ and the point(s) $(u^*,v^*)$ with $v^*=u^*$ and where $u^*$ is determined by the solution(s) of
\begin{align}
\label{htae1}
\begin{aligned}
& (u+A)(1-u)(u-M)=Qu \,, \quad\text{or equivalently}\,,\\
& d(u)=u^3-(M+1-A)u^2-(A(M+1)-Q-M)u+AM =0\,.
\end{aligned}
\end{align}
Note that the equilibrium point at the origin in \eqref{hta2} corresponds a singular point in \eqref{hta1}. Define the functions $g(u)=(u+A)(1-u)(u-M)$ and $h(u)=Qu$ and observe that $\lim\limits_{u \rightarrow \pm \infty} g(u)= \mp \infty$ and $g(0)=-AM<0$. So, \eqref{htae1} will always have a single negative real root, which we denote by $u=-H$ where $H>0$, see Figure \ref{Fig.eqpoint}. From \eqref{htae1}, we now get that $Q=(H+1)(H+M)(A-H)/H$, and given that $Q>0$, we conclude that $A>H$. Factoring out $(u+H)$ from \eqref{htae1} leaves us with second order polynomial
\begin{align}
u^2-(H+M+1-A)u+\dfrac{AM}{H}=0\,,\label{htae3}	
\end{align} where $H+M+1-A>0$ since $A<1$. The roots of \eqref{htae3} are given by 
\begin{equation}\label{delta}
u_{1,2} = \frac12 \left(H+M+1-A \pm \sqrt{\Delta} \right)\,, \quad {\rm with}
\,\, \Delta=(H+M+1-A)^2-\frac{4AM}{H}\,,
\end{equation}
such that $M<u_1\leq E \leq u_2<1$, where $E=(H+M+1-A)/2$.

Modifying the parameter $Q$ impacts $\Delta$ (since $H$ depends on $Q$) and hence the number of positive equilibrium points. In particular, 
\begin{enumerate}[label=\roman*.]
\item if $\Delta<0$, then \eqref{hta2} has no equilibrium points in the first quadrant;
\item if $\Delta>0$, then \eqref{hta2} has two equilibrium points $P_{1,2}=(u_{1,2},u_{1,2})$ in the first quadrant; and
\item if $\Delta=0$, then \eqref{hta2} has one equilibrium point $(E,E)$ of order two in the first quadrant,
\end{enumerate}
see also Figure \ref{Fig.eqpoint}. Finally, observe that none of these equilibrium points explicitly depend on the system parameter $S$. 
Therefore, $S$ and $Q$ are the natural candidates to act as bifurcation parameters.

\begin{figure}
\centering
\includegraphics[width=7.5cm]{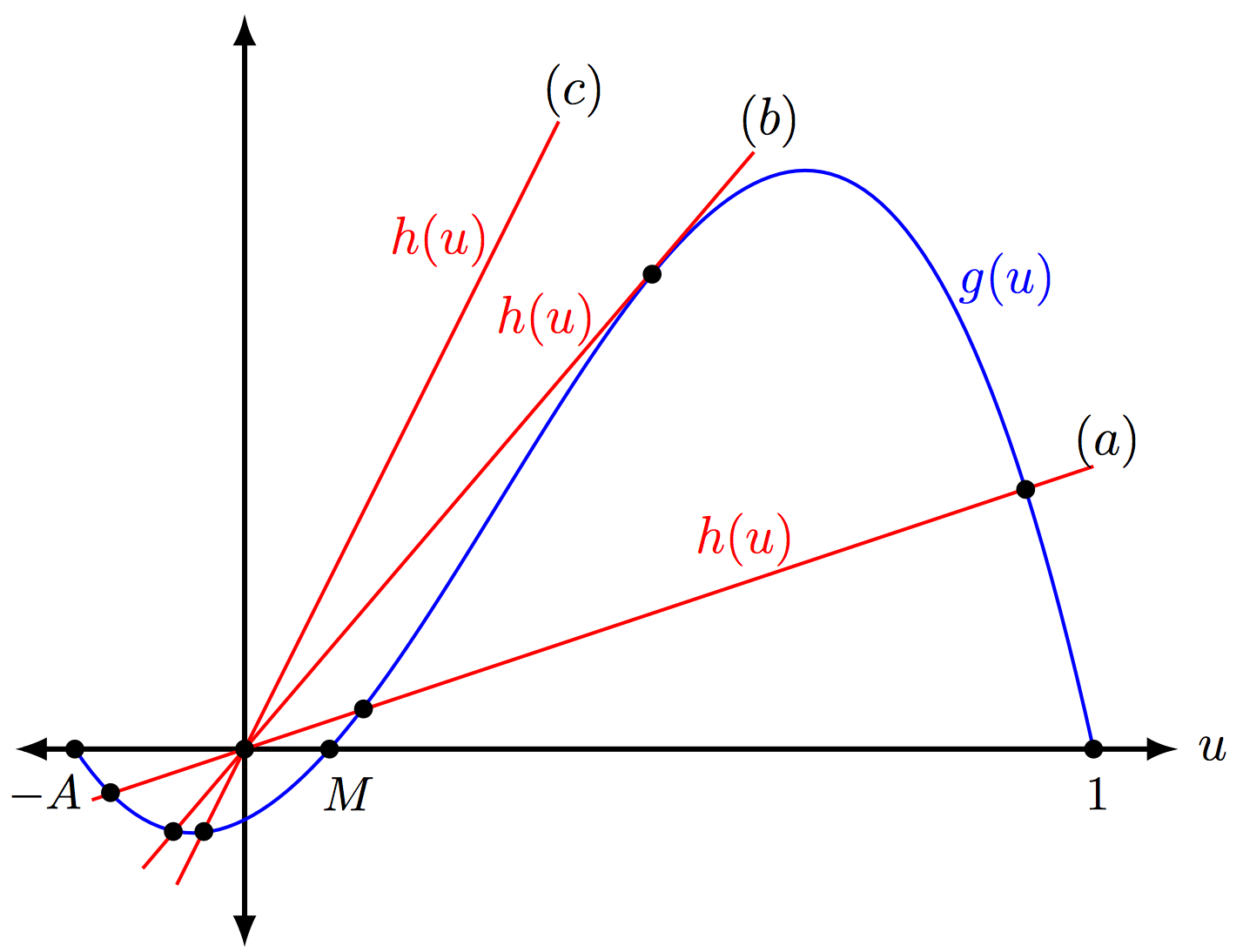}
\caption{The intersections of the function $g(u)$ (blue line) and the straight line $h(u)$ (red lines) for the three possible cases: (a) for $\Delta>0$ \eqref{delta} there are two distinct intersections (corresponding to the existence of two equilibrium points in the first quadrant); (b) for $\Delta=0$ there is a unique intersection (corresponding to the existence of a unique equilibrium point of order two); and (c) for $\Delta<0$ there are no intersections.}
\label{Fig.eqpoint}
\end{figure}

\section{Main Results}\label{SS:RES}

In this section, we discuss the stability of the equilibrium points of system \eqref{hta2} and their basins of attraction. 
\begin{theorem} \label{the2}
All solutions of \eqref{hta2} which are initiated in the first quadrant are bounded and eventually end up in $\Phi=\{(u,v),\ 0\leq u\leq1,\ 0\leq v\leq1\}$.
\end{theorem}
\begin{proof}
First, observe that all the equilibrium points lie inside of $\Phi$. Additionally, as the system is of Kolmogorov type, the $u$-axis and  $v$-axis are invariant sets of \eqref{hta2}. Moreover, the set $\Gamma=\{(u,v),\ 0\leq u\leq1,\  v\geq0\}$ is an invariant region since $du/d\tau \leq0$ for $u=1$ and $v\geq0$. 
That is, trajectories entering into $\Gamma$ remain in $\Gamma$, see Figure \ref{Fig.bc}. 
If $u>1$ and $u>v$ we have that $du/d\tau < 0$ and $dv/d\tau>0$. Hence, trajectories inside this region enter into $\Phi$ or the region where $u>1$ and $v\geq u$, see $\Lambda$ and $\Theta$ in Figure \ref{Fig.bc}. 
If $u>1$ and $v \geq u$, we have $du/d\tau<0$ and $dv/d\tau \leq 0$. Hence, both the $u$-component and $v$-component of trajectories inside this region are non-increasing as time increases and enter into $\Gamma$ with $v>1$, see $\Theta$ in Figure \ref{Fig.bc}. Thus, all trajectories starting outside $\Gamma$ enter into $\Gamma$. Finally, for all $(u,v)\in\Gamma\backslash\Phi$  we have that $dv/d\tau<0$. Therefore, all trajectories end up in $\Phi$.
\end{proof}

\begin{figure}
\centering
\includegraphics[width=8cm]{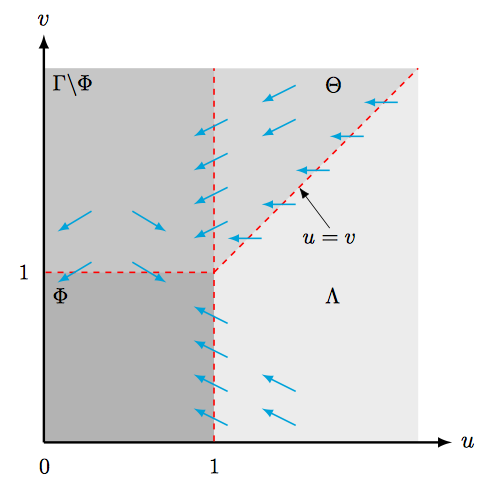}
\caption{Phase plane of system \eqref{hta2} and its invariant regions $\Phi$ and $\Gamma$.}
\label{Fig.bc}
\end{figure}

\subsection{The nature of the equilibrium points}\label{nateq}

To determine the nature of the equilibrium points we compute the Jacobian matrix $J(u,v)$ of \eqref{hta2}
\[J(u,v)=\begin{pmatrix}
u(ug'(u)+2(g(u)-Qv))  & -Qu^2 \\ 
Sv( A+2u-v)  &  S(u-2v)( A+u) 
\end{pmatrix},\]
where 
\begin{align}
\label{deri}
\begin{aligned}
& g(u)=(u+A)(1-u)(u-M) \, \quad\text{and}\,\\
& g'(u)=(1-u)(u-M)+(u+A)(1-u)-(u+A)(u-M)\,.
\end{aligned}
\end{align}
\begin{lemma}\label{eqax}
The equilibrium point $(1,0)$ is a saddle point and $(M,0)$ is a hyperbolic repeller.
\end{lemma}
\begin{proof}
The Jacobian matrix evaluated at $(1,0)$ gives
\[ J(1,0)=\begin{pmatrix}
-(1-M)(A+1)  & -Q \\ 
0  &  S(A+1) 
\end{pmatrix},\]
which has eigenvalues \[\lambda_{(1,0)}^1=-(1-M)(A+1)<0\quad\text{and}\quad\lambda_{(1,0)}^2=S(A+1)>0\] and eigenvectors \[\psi_{(1,0)}^1=\left( 1 , 0 \right) ^T\quad\text{and}\quad\psi_{(1,0)}^2=\left( \dfrac{-Q}{(1-M+S)(A+1)},1 \right) ^T.\] Similarly, the Jacobian matrix evaluated at $(M,0)$ gives
\[J(M,0)=\begin{pmatrix}
M^2(1-M)(A+M)  & -QM^2 \\ 
0  &  MS(A+M) 
\end{pmatrix},\]
which has eigenvalues \[\lambda_{(M,0)}^1=M^2(1-M)(A+M)>0\quad\text{and}\quad\lambda_{(M,0)}^2=MS(A+M)>0\] and eigenvectors \[\psi_{(M,0)}^1=\left(1,0\right)^T\quad\text{and}\quad\psi_{(M,0)}^2=\left( \dfrac{MQ}{(M(1-M)-S)(A+M)}, 1 \right)^T.\]
		
Since $M<1$, it follows that $(1,0)$ is a saddle point and $(M,0)$ is a hyperbolic repeller in system \eqref{hta2}.
\end{proof}

\begin{lemma}\label{lemm2}
The origin $(0,0)$ in system \eqref{hta2} is a non-hyperbolic attractor. Furthermore, if there are no positive equilibrium points in the first quadrant, i.e. for $\Delta<0$ \eqref{delta}, then $(0,0)$ is globally asymptotically stable in the first quadrant.
\end{lemma}

\begin{proof}
First, we observe that after setting $u=0$ in system \eqref{hta2} the second equation becomes $dv/dt=-v^2AS<0$ for $v>0$. That is, any trajectory starting along the positive $v$-axis converges to the origin $(0,0)$. The Jacobian matrix evaluated at the origin reduces to the zero matrix. Hence, the origin $(0,0)$ is a non-hyperbolic equilibrium point of system \eqref{hta2}. We use the \emph{blow-up} method to desingularised the origin and study the dynamics of this blown-up equilibrium point. More specifically, we consider the method used in Dumortier {\em et al.\ } \cite{dumortier} and introduce, with slight abuse of notation, the transformation
\begin{equation}
(u,v)\to (xy,y)\,,\label{eq:vbu2}
\end{equation}and the time rescaling
\begin{equation}\label{eq:cvtr2}
\tau \to \dfrac{t}{y}\,.
\end{equation}
We omitted the {\emph{blow-up in the $x$-direction}} \cite{dumortier} since it does not give any further information. Transformation \eqref{eq:vbu2} is well defined for all values of $u$ and $v$ except for $v=0$ and blows-up the origin into the entire $x$-axis. In these new coordinates, \eqref{hta2} becomes
\begin{equation}\label{eq:0sys2}
\begin{aligned}
\dfrac{dx}{dt}&=x \left(S (1-x) (A + x y) + x (M - x y) (x y-1) (A + x y)-Q x y)\right)\,,\\
\dfrac{dy}{dt}&=S (x-1)(xy+A)y\,.
\end{aligned}
\end{equation}
System \eqref{eq:0sys2} has two equilibrium points on the nonnegative $x$-axis: the origin $(0,0)$ and a second equilibrium point $(\mu,0)$ with $\mu=S/(S+M)$. The corresponding Jacobian matrix of \eqref{eq:0sys2} evaluated at $(0,0)$ gives
\begin{equation*} 
J(0,0)=\begin{pmatrix}
AS & 0 \\
0     & -AS
\end{pmatrix},
\end{equation*} 
which has eigenvalues $\lambda_{(0,0)}^1=AS>0$ and $\lambda_{(0,0)}^2=-AS<0$, and corresponding eigenvectors $\phi_{(0,0)}^1=(0,1)^T$ and $\phi_{(0,0)}^2=(1,0)^T$. Hence, the origin is a saddle point in system \eqref{eq:0sys2}. Similarly, the corresponding Jacobian matrix of \eqref{eq:0sys2} evaluated at $(\mu,0)$ gives
\begin{equation*}
J(\mu,0)=\begin{pmatrix}
-AS &  \dfrac{S^2(AS(1+M)-Q(M+S)}{(M+S)^3}\\[3mm]
0 &- \dfrac{AMS}{M+S}
\end{pmatrix},
\end{equation*}
with eigenvalues $\lambda_{(\mu,0)}^1=-AS<0$ and $\lambda_{(\mu,0)}^2=-AMS/(M+S)<0$, and corresponding eigenvectors $\phi_{(\mu,0)}^1=\left( (AS(1+M)-Q(M+S))/(A(M+S)^2), 1\right) ^T$ and $\phi_{(\mu,0)}^2=(1,0)^T$. Hence, $(\mu,0)$ is an attractor in system \eqref{eq:0sys2}. A branch of the stable eigenvector $\phi_{(\mu,0)}^1$ is in the half-plane $y>0$, as illustrated in the left panel of Figure \ref{fig:022}, while the other eigenvectors are the axes. 
	
Taking the  inverse of \eqref{eq:vbu2}, the line $y=0$, including the equilibrium point $(\mu,0)$, collapses to the origin $(0,0)$ of \eqref{hta2}, the line $x=0$ is mapped to the line  $u=0$, and $\phi_{(\mu,0)}^1$ is locally mapped to a curve $\Gamma^s$, see Figure \ref{fig:022}. Since the orientation of the orbits in the first quadrant is preserved by  \eqref{eq:vbu2} and \eqref{eq:cvtr2}, it follows that the origin $(0,0)$ is a non-hyperbolic attractor of \eqref{hta2}.
	
Finally, by Theorem~\ref{the2} we have that solutions starting in the first quadrant are bounded and eventually end up in the invariant region $\Gamma$. Moreover, the equilibrium point $(1,0)$ is a saddle point and, if $\Delta<0$ \eqref{delta}, there are no equilibrium points in the interior of the first quadrant. Thus, by the Poincar\'e--Bendixson Theorem the unique $\omega$-limit of all the trajectories is the origin, see Figure \ref{fig:04}.
\end{proof}
	
\begin{figure}
\centering
\includegraphics[width=12cm]{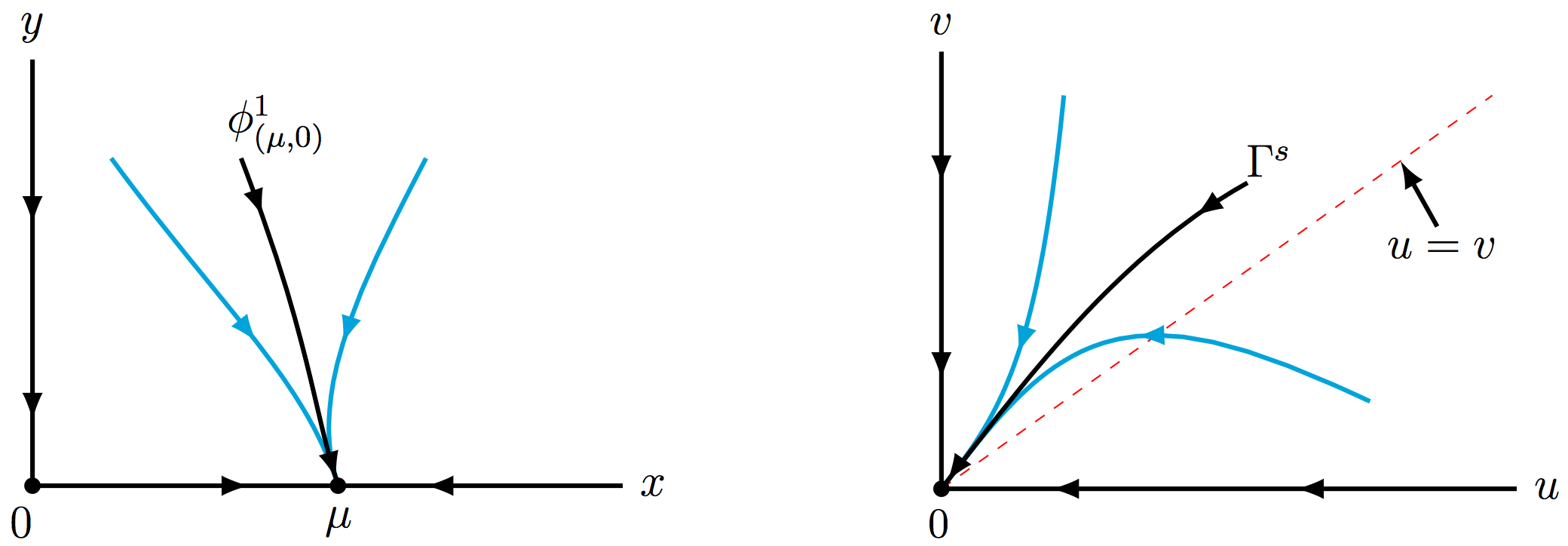}
\caption{The panel on the left illustrates the blow-up in the $y$-direction of the origin. The panel on the right illustrates the blow-down of this blow up.}
\label{fig:022}
\end{figure}

\begin{figure}
\centering
\includegraphics[width=6.5cm]{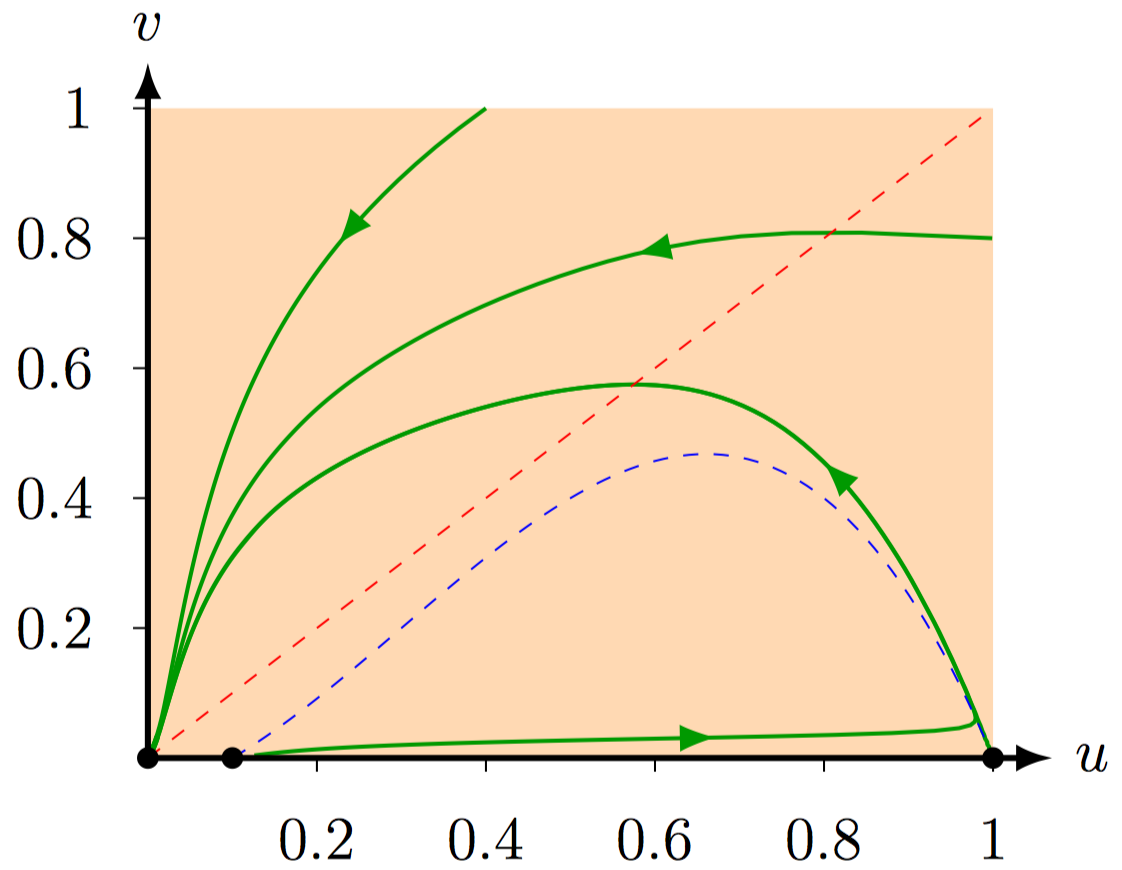}
\caption{For $M=0.1$, $A=0.2$, $Q=0.35$, and $S=0.1$, such that $\Delta<0$ \eqref{delta}, the equilibrium point $(0,0)$ is a global attractor. The dashed blue (red) curve represents the prey (predator) nullcline.}
\label{fig:04}
\end{figure}

\subsubsection{$\Delta > 0$}
Next, we consider the stability of the two positive equilibrium points $P_{1,2}$ of system \eqref{hta2} in the interior of $\Phi$. These equilibrium points lie on the curve $u=v$ such that $g(u)=Qu$ \eqref{deri},  and they only exist if the system parameters are such that $\Delta>0$ \eqref{delta}. The Jacobian matrix of system \eqref{hta2} at these equilibrium points becomes
\begin{align}
\label{JP1}
J(P_{1,2})=\begin{pmatrix}
u_{1,2}^2g'(u_{1,2})   & -Qu_{1,2}^2 \\ 
Su_{1,2}(A+u_{1,2})  &  -Su_{1,2}(A+u_{1,2}) 
\end{pmatrix}.\end{align}
The determinant of $J(P_{1,2})$ is given by
\begin{align*}
\det(J(P_{1,2})) =&Su_{1,2}^3(A+u_{1,2})(Q-g'(u_{1,2}))\\
=&-Su_{1,2}^3(A+u_{1,2})((-H+M+1-A)u_{1,2}-2u_{1,2}^2\\
&+H(H+M+1-A))\\
=&Su_{1,2}^3(A+u_{1,2})(H+u_{1,2})(-H-M-1+A+2u_{1,2}),
\end{align*}

where we used \eqref{htae1} in the first step and \eqref{htae3} in the last step. Additionally, the trace of the Jacobian  matrix is
\begin{align*}
\begin{aligned}
\tr(J(P_{1,2})) &=u_{1,2}(u_{1,2}g'(u_{1,2})-S(A+u_{1,2}))\\
&=u_{1,2}(u_{1,2}+A)\left( \dfrac{u_{1,2}g'(u_{1,2})}{(A+u_{1,2})}-S\right)\\
&=u_{1,2}(u_{1,2}+A)\left( f(u_{1,2})-S\right),
\end{aligned}
\end{align*}
where 
\begin{equation}\label{eq:fu}
f(u_{1,2}):=\dfrac{u_{1,2}g'(u_{1,2})}{(A+u_{1,2})}.
\end{equation} 
Thus, the sign of the determinant depends on the sign of $-H-M-1+A+2u_{1,2}$, while the sign of the trace depends on the sign of $f(u_{1,2})-S$. This gives the following results.
\begin{lemma}\label{p1}
Let the system parameters of \eqref{hta2} be such that $\Delta>0$ \eqref{delta}. Then, the equilibrium point $P_1$ is a saddle point.
\end{lemma}
\begin{proof}
Evaluating $-H-M-1+A+2u$ at $u_1$ gives:\\
\[\begin{aligned}
-H-M-1+A+2u_1&=-H-M-1+A+(H+M+1-A-\sqrt{\Delta})\\
&=-\sqrt{\Delta}<0.
\end{aligned}\]
Hence, $\det(J(P_1))<0$ and $P_1$ is thus a saddle point.
\end{proof}
\begin{lemma}\label{lemm4}
Let the system parameters be such that $\Delta>0$ \eqref{delta}. Then, the equilibrium point $P_2$ is:
\begin{enumerate}[label=\roman*.]
\item a repeller if $0<S<S^*:=f(u_2)$; and
\item an attractor if $S>S^*$,
\end{enumerate}
with $f$ defined in \eqref{eq:fu}. 
\end{lemma}
\begin{proof}
Evaluating $-H-M-1+A+2u$ at $u=u_2$ gives\\
\[\begin{aligned}
-H-M-1+A+2u_2&=-H-M-1+A+(H+M+1-A+\sqrt{\Delta})\\
&=\sqrt{\Delta}>0
\end{aligned}.\]
Hence, $\det(J(P_2))>0$. Evaluating $f(u)-S$ at $u=u_2$ gives
\[\begin{aligned}
f(u_2)-S=& \dfrac{u_{2}g'(u_{2})}{(A+u_{2})}-S,\\
\end{aligned}\]
with $g'(u_2)$ defined in \eqref{deri}. Therefore, the sign of the trace, and thus the behaviour of $P_2$ depends on the parity of $f(u_2)-S$ (and recall that $u_2$ does not depend explicitly on $S$). 
\end{proof}  

In conclusion, for system parameters $(Q,M,A)$ such that $\Delta>0$ and for $S>S^*$, system \eqref{hta2} has two attractors, namely $(0,0)$ and $P_2$. Furthermore, at the critical value $S=S^*$, such that $\tr(J(P_2))=0$, $P_2$ undergoes a Hopf bifurcation \cite{chicone}. Note that $S^*$ depends on $Q$ and it can actually be negative. In that  case $P_2$ is an attractor for all $S>0$ (and as long as $\Delta>0$), see also Figures~\ref{fig:03} and \ref{summ}.

Next, we discuss the basins of attraction of the attractors $(0,0)$ and $P_2$ (for $S>S^*$) in $\Phi$ (see Theorem~\ref{the2}). The stable manifold of the saddle point $P_1$, $W^s(P_1)$, often acts as a separatrix curve between these two basins of attraction. The eigenvalues of the associated Jacobian matrix of $P_1$ \eqref{JP1} are given by
\[\lambda_{P_1}^{u,s}   = \dfrac{u_1(A+u_1)}{2}\left(f(u_1)-S\pm \dfrac{\sqrt{p(u_1)}}{(A+u_1)}\right)\,,\]
where $p(u_1)=(A+u_1)(f(u_1)+2Su_1(f(u_1)-2Q)+(S^2+f^2(u_1))(u_1+A))$ and the (un)stable eigenvectors
\[\psi_{P_1}^{u,s}=\left( \left( 1+\dfrac{\lambda_{P_1}^{u,s}}{Su_1(A+u_1)}\right) ,1\right) ^T\,.\]
Let $W^{u,s}_{\nearrow}(P_1)$ be the (un)stable manifold of $P_1$ associated to the eigenvector $\psi_{P_1}^{u,s}$ that goes up to the right (from $P_1$) and let $W^{u,s}_{\swarrow}(P_1)$ be the (un)stable manifold of $P_1$ associated to the eigenvector $\psi_{P_1}^{u,s}$ that goes down to the left (from $P_1$), see Figure \ref{fig:man}. From the phase plane and the nullclines of system \eqref{hta2} it immediately follows that $W^s_{\nearrow}(P_1)$ is connected with $(M,0)$ and $W^u_{\swarrow}(P_1)$ with $(0,0)$. Furthermore, everything in between of $W^{s}_{\nearrow}(P_1)$, $W^{u}_{\swarrow}(P_1)$ and the $u$-axis also asymptotes to the origin, see Figure \ref{fig:man}. 

\begin{figure}
\centering
\includegraphics[width=7cm]{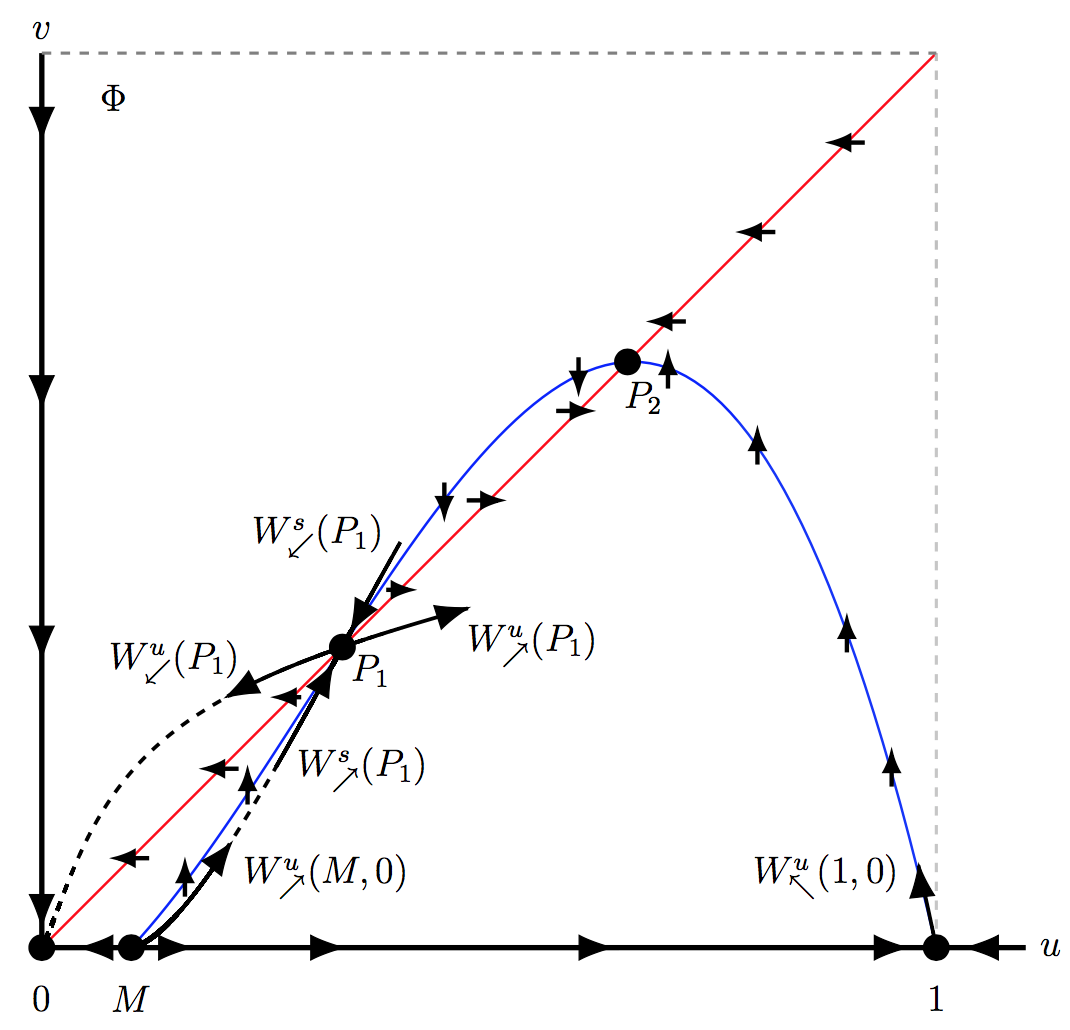}
\caption{Stable and unstable manifolds of $P_1$.}
\label{fig:man}
\end{figure}

For $\Delta>0$ and depending on the value of $S$, there are six different cases for the boundary of the basins of attraction in the invariant region $\Phi$, see Theorem~\ref{the2}. By continuity of the vector field in $S$, see \eqref{hta2}, we get:
\begin{enumerate}[label=\roman*.]
\item For $0<S\leq S^*=f(u_2)$, the equilibrium point $P_2$ is unstable, see lemma \ref{lemm4}, and $W^s_{\swarrow}(P_1)$ connects with $P_2$. Hence, $\Phi$ is the basin of attraction of $(0,0)$, see Figure \ref{fig:07}.
	
\begin{figure}
\centering
\includegraphics[width=6.5cm]{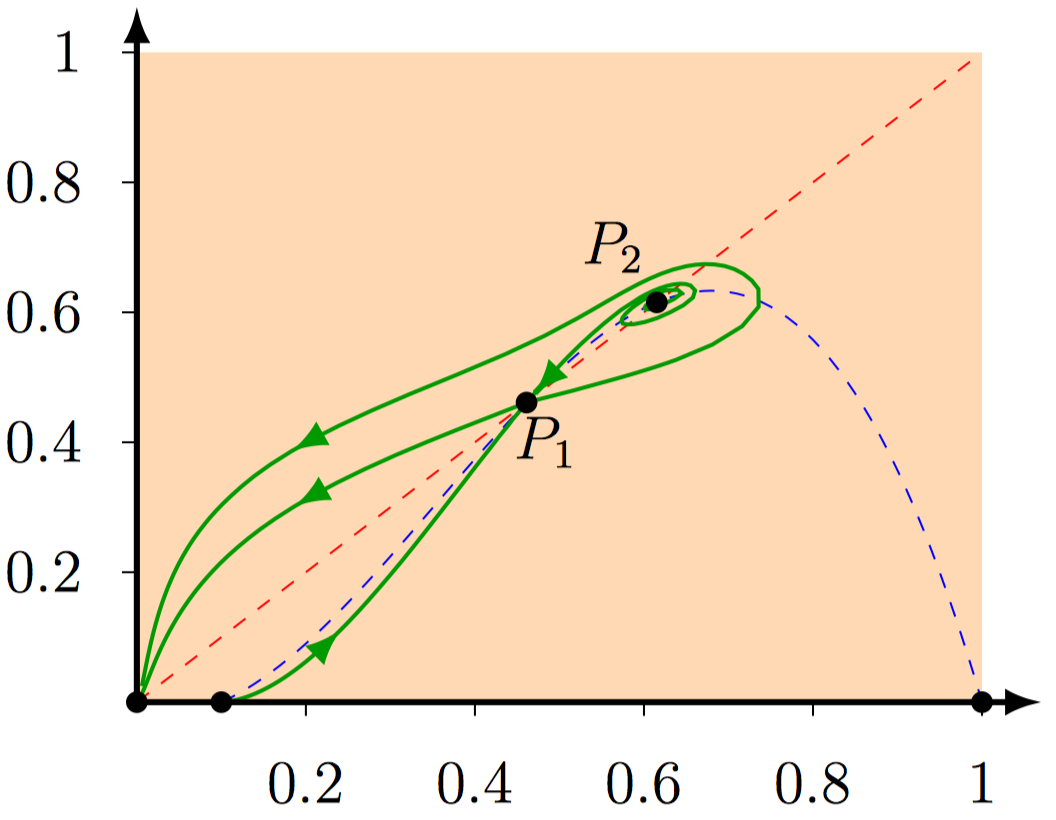}
\caption{For $M=0.1$, $A=0.0365$, $Q=0.21$, and $S=0.0123$, such that $\Delta>0$ \eqref{delta} and $S<S^*$, the equilibrium point $(0,0)$ is a global attractor. The blue (red) curve represents the prey (predator) nullcline.}
\label{fig:07}
\end{figure}

\item For $S^*<S<S^{**}$, there is an unstable limit cycle that surrounds $P_2$ and $W^s_{\swarrow}(P_1)$ connects with this limit cycle. This limit cycle is created around $P_2$ via the Hopf bifurcation \cite{gaiko} and terminates via a homoclinic bifurcation at $S=S^{**}$. 
Therefore, the limit cycle acts as a separatrix curve between the basins of attraction of $P_2$ and $(0,0)$ in this parameter regime, see Figure \ref{fig:06}(a).
\item  For $S=S^{**}$, $W^s_{\swarrow}(P_1)$ connects with $W^u_{\nearrow}(P_1)$ generating an homoclinic curve. This curve is the separatrix curve between the basins of attraction of $(0,0)$ and $P_2$ ($P_2$ is an attractor since $S>S^*$), see Figure \ref{fig:06}(b).

\begin{figure}
\centering
\includegraphics[width=12cm]{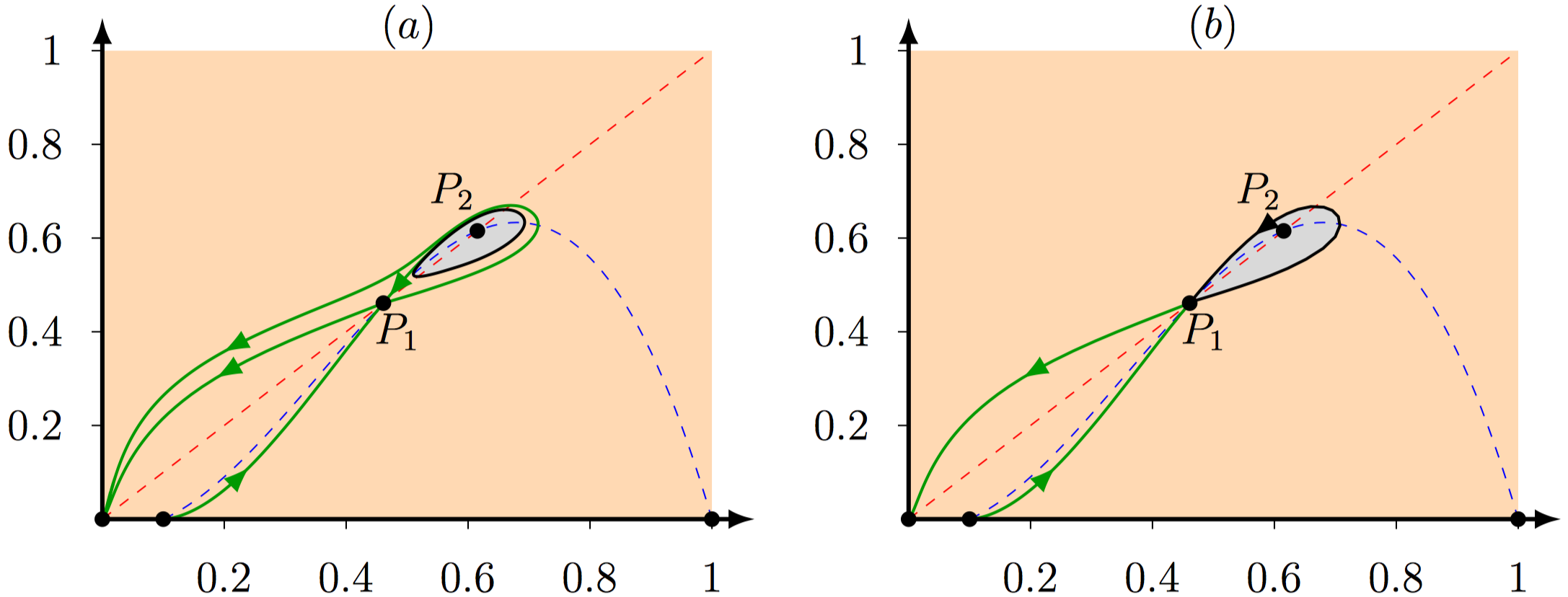}
\caption{For $M=0.1$, $A=0.0365$, and $Q=0.21$, such that $\Delta>0$ \eqref{delta} and $S^*<S\leq S^{**}$, the equilibrium point $(0,0)$ and $P_2$ are both attractors. (a) For $S^*<S=0.121<S^{**}$ the equilibrium point $P_2$ is surrounded by a unstable limit cycle (black curve) that forms the separatrix curve between the basin of attraction of $P_2$ (grey region) and the basin of attraction of $(0,0)$ (orange region). 
(b) For $S=S^{**}=0.148$ the stable and unstable manifold of $P_1$ generate a homoclinic curve (black curve) that forms the separatrix curve between the basins of attraction. The blue (red) curve represents the prey (predator) nullcline. Observe that the same color conventions are used in the upcoming figures.}
\label{fig:06}
\end{figure}

\item  For $S^{**}<S<S^{***}$, both $W^s_{\swarrow}(P_1)$ and $W^s_{\nearrow}(P_1)$ connect with the equilibrium point $(M,0)$ forming two heteroclinic curves. These curves also form the separatrix curves between the stable equilibrium points in $\Phi$. In other words, $W^s(P_1)$ is the separatrix curve, see Figure \ref{fig:08a}(a). 
\item For $S=S^{***}$, $W^s_{\swarrow}(P_1)$ connects with $(1,0)$ generating a heteroclinic curve, which, together with $W^s_{\nearrow}(P_1)$, form the separatrix curves of the basins of attraction in $\Phi$, see Figure \ref{fig:08a}(b).

\begin{figure}
\centering
\includegraphics[width=12cm]{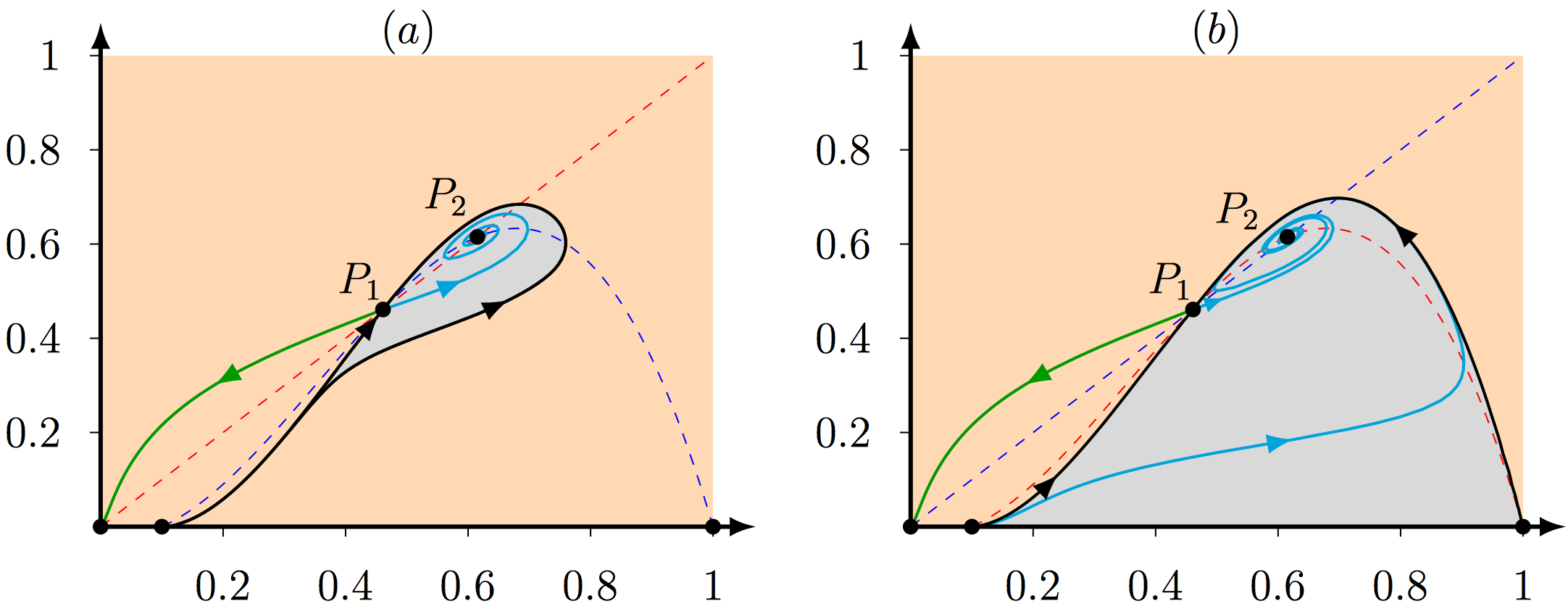}
\caption{For $M=0.1$, $A=0.0365$, and $Q=0.21$, such that $\Delta>0$ and $S^{**}<S\leq S^{***}$, the equilibrium points $(0,0)$ and $P_2$ are attractors. 
(a) For $S^{**}<S=0.16<S^{***}$ both branches of the stable manifold of $P_1$ connect with the equilibrium point $(M,0)$ and $W^s(P_1)$ forms the separatrix curve between the basins of attraction.
(b) For $S=S^{***}=0.1915$ the stable manifold of $P_1$ connects with the unstable manifold of $(1,0)$ and the equilibrium point $(M,0)$, again forming the separatrix curve.  (See Figure \ref{fig:06} for the color conventions.)}
\label{fig:08a}
\end{figure}

\item For $S>S^{***}$, $W^s_{\swarrow}(P_1)$ intersects the boundary of $\Phi$, and $W^s(P_1)$ again forms the separatrix curve in $\Phi$, see Figure \ref{fig:08}.

\begin{figure}
\centering
\includegraphics[width=6.5cm]{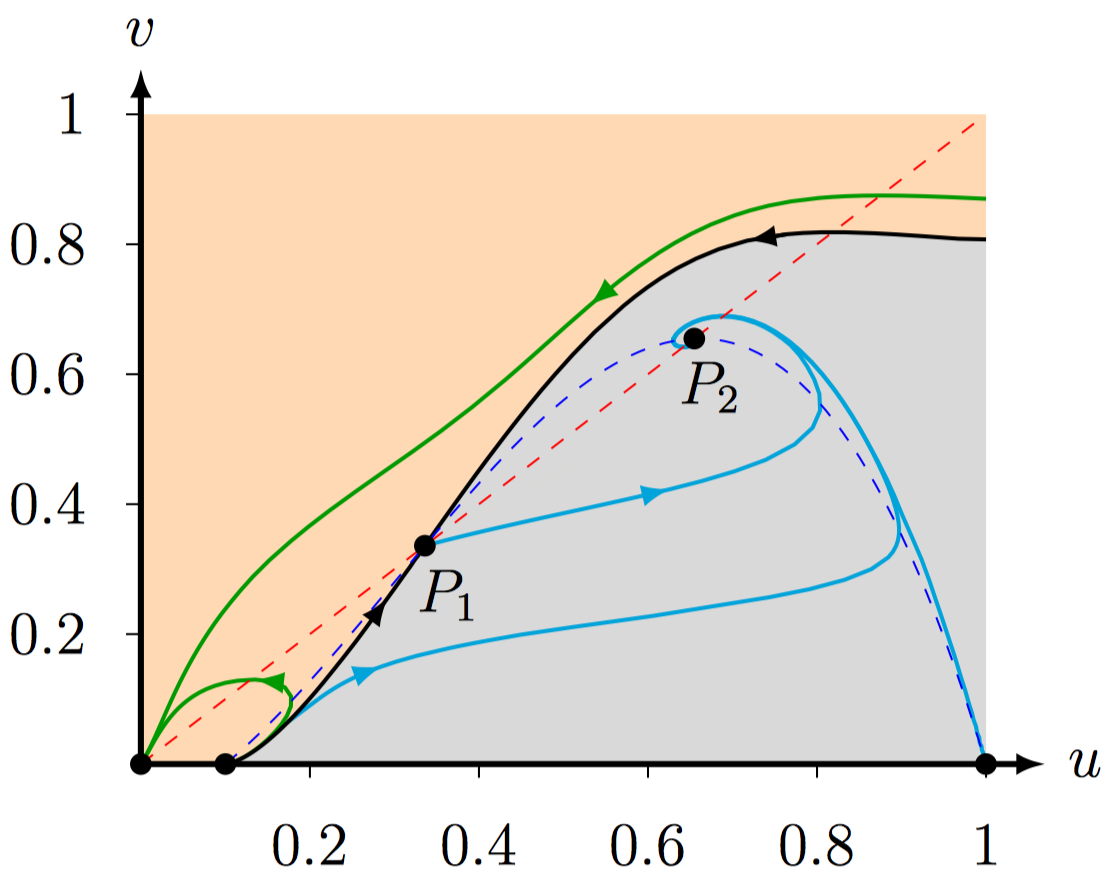}
\caption{For $M=0.1$, $A=0.0365$, $Q=0.21$, and $S=0.1$, such that $\Delta>0$ and $S>S^{***}$,  the equilibrium points $(0,0)$ and $P_2$ are attractors and $W^s(P_1)$ forms the separatrix curve in $\Phi$. (See Figure \ref{fig:06} for the color conventions.)}
\label{fig:08}
\end{figure}
\end{enumerate}
Note that the system parameters $(Q, M, A)$ are fixed at $(0.21, 0.1, 0.0365)$ in Figures~\ref{fig:07} - \ref{fig:08}. Consequently, $u_{1,2}$ and $H$ are also constant. In particular, $u_1 \approx 0.4611, u_2 \approx 0.6153$ and $H \approx 0.01287$.
Observe that we described the most involved situation above and that there are actually parameter values for which $S^*$ is negative. Furthermore, there are also parameter values for which $S^{**}$ and $S^{***}$ do not exist. In these instances, there are fewer cases for increasing $S$ and the observed change in behavior will be a subset of what is described above, see also Figures~\ref{fig:03} and \ref{summ}

\subsubsection{$\Delta=0$}

Finally, if $\Delta=0$ \eqref{delta} the equilibrium points $P_1$ and $P_2$ collapse such that  $u_1=u_2=E=(H+M+1-A)/2$. Therefore, system \eqref{hta2} has one equilibrium point of order two in the first quadrant given by $(E,E)$.
\begin{lemma} \label{lemm5}
Let the system parameters be such that $\Delta=0$ \eqref{delta}. Then, the equilibrium point $(E,E)$ is:
\begin{enumerate}[label=\roman*.]
\item a saddle-node attractor if $S<f(E)=\dfrac{Q(H+M+1-A)}{H+M+1+A}$\,,
\item a saddle-node repeller if $S>f(E)$\,.
\end{enumerate}
\end{lemma}
See Figure \ref{fig:saddlenode} for phase portraits related to both cases of Lemma~\ref{lemm5}.

\begin{figure}
\centering
\includegraphics[width=12cm]{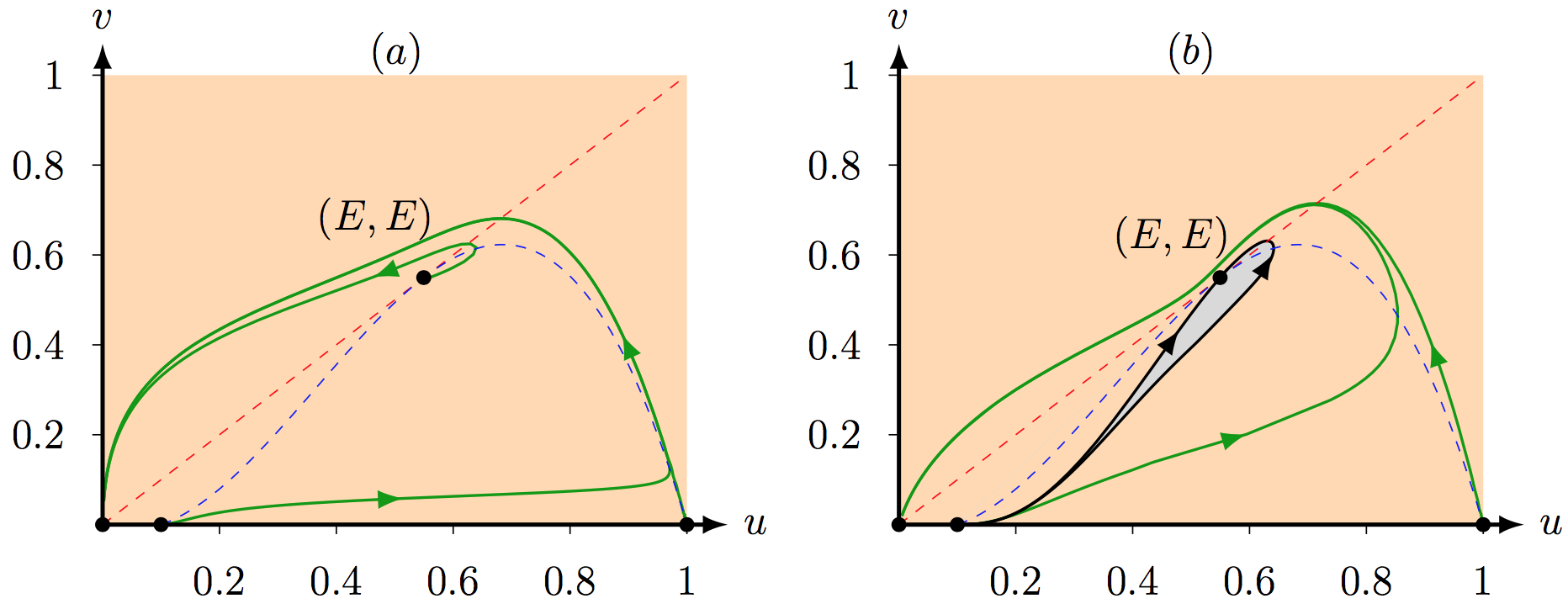}
\caption{For $M=0.1$, $A=0.0009$, and $Q=0.20283145$, such that $\Delta=0$, system \eqref{hta2} has one equilibrium point $(E,E)$ of order two. (a) For $S<f(E)$ the equilibrium point $(E,E)$ is a saddle-node repeller. (b) For $S>f(E)$ the equilibrium point $(E,E)$ is a saddle-node attractor. (See Figure \ref{fig:06} for the color conventions.)}
\label{fig:saddlenode}
\end{figure}

\begin{proof}
Evaluating $-H-M-1+A+2u$ at $u=E$ gives    
\begin{align*}
-H-M-1+A+2u&=-H-M-1+A+(H+M+1-A)=0\,.
\end{align*}
Therefore, $\det(J(E,E))=0$. Next, evaluating $f(E)-S$, where $f$ is defined in \eqref{eq:fu}, gives
\begin{align*}
f(E)-S =&\dfrac{Q(H+M+1-A)}{H+M+1+A}-S.
\end{align*}
Therefore, the behaviour of the equilibrium point $(E,E)$ depends on the parity of $Q(H+M+1-A)/(H+M+1+A)-S$ and recall that $H$ does not depend on $S$. 
\end{proof}

\subsection{Bifurcation Analysis}
In this section, we discuss some of the possible bifurcation scenarios of system \eqref{hta2}.

\begin{theorem}\label{the5}
Let the system parameters be such that $\Delta=0$ \eqref{delta}. Then, system \eqref{hta2} experiences a saddle-node bifurcation at the equilibrium point $(E,E)$ (for changing $Q$).
\end{theorem}
\begin{proof}
The proof of this theorem is based on Sotomayor's Theorem \cite{perko}.  
For $\Delta=0$, there is only one equilibrium point $(E,E)$ in the first quadrant, with $E=(H+M+1-A)/2$. 
From the proof of Lemma~\ref{lemm5} we know that $\det(J(E,E))=0$ if $\Delta=0$. Additionally, let $U=(1,1)^T$ the eigenvector corresponding to the eigenvalue $\lambda=0$ of the Jacobian matrix $J(E,E)$, and let 
$$W=\left(-\frac{S(H+M+1+A)}{Q(H+M+1-A)},1\right)^T$$ be the eigenvector corresponding to the eigenvalue $\lambda=0$ of the transposed Jacobian matrix $J(E,E)^T$. 
	
If we represent \eqref{hta2} by its vector form
\begin{align} \nonumber %\label{eq:vf}
F(u,v;Q) =\begin{pmatrix}
(u+A)(1-u)(u-M)-Qv\\ 
u-v
\end{pmatrix},
\end{align}
then differentiating $F$ at $(E,E)$
with respect to the bifurcation parameter $Q$ gives
\[F_Q(E,E;Q)=\begin{pmatrix}
-\dfrac{1}{2}(H+M+1-A)\\ 
0
\end{pmatrix}.\]
Therefore,
\[W \cdot F_Q(E,E;Q)=\dfrac{S(H+M+1+A)}{2Q}\neq0.\]
Next, we analyse the expression $W \cdot [D^2F(E,E;Q)(U,U)]$.
Therefore, we first compute the Hessian matrix $D^2F(u,v;Q)(V,V)$, where $V=(v_1,v_2)$,	
\[\begin{aligned}
D^2F(u,v;Q)(V,V) =& \dfrac{\partial^2F(u,v;Q)}{\partial u^2}v_1v_1+\dfrac{\partial^2F(u,v;Q)}{\partial u\partial v}v_1v_2+\dfrac{\partial^2F(u,v;Q)}{\partial v\partial u}v_2v_1\\
&+\dfrac{\partial^2f(u,v;Q)}{\partial v^2}v_2v_2\,.
\end{aligned}\]
At the equilibrium point $(E,E)$ and $V=U$, this simplifies to
\[\begin{aligned}	
D^2F(E,E;Q)(U,U)& = \begin{pmatrix}
2(M-2-A)\\ 
0
\end{pmatrix}\,. 
\end{aligned}\]
Hence, since $M \in (0,1)$ and $A \in (0,1)$, we get
\[\begin{aligned}
W \cdot [D^2F(E,E;Q)(U,U)]= -\dfrac{2S(H+M+1+A)(M-2-A)}{Q(H+M+1-A)}\neq0 \,.
\end{aligned}\]
By Sotomayor's Theorem \cite{perko} it now follows that system \eqref{hta2} has a saddle-node bifurcation at the equilibrium point $(E,E)$.
\end{proof}

\begin{theorem} \label{the4}
Let the system parameters be such that $\Delta=0$ \eqref{delta} and $S=f(E)$ \eqref{eq:fu}. Then, system \eqref{hta2} experiences a Bogdanov--Takens bifurcation at the equilibrium point $(E,E)$ (for changing $(Q,S)$).
\end{theorem}
\begin{proof}
If $\Delta=0$, or equivalently $Q=g'(E)$, and $f(E)=S$, then the Jacobian matrix of system \eqref{hta2} evaluated at the equilibrium point $(E,E)$ simplifies to
\[\begin{aligned}
J(E,E) &=\begin{pmatrix}
S(E+A)E & -S(E+A)E \\ 
S(E+A)E & -S(E+A)E 
\end{pmatrix},\\
&=\dfrac{S}{4}(H+M+1-A)(H+M+1+A) \begin{pmatrix}
1 & -1 \\ 
1 & -1 
\end{pmatrix}. 
\end{aligned}\]
So, $\det(J(E,E))=0$ and $\tr(J(E,E))=0$.
Next, we find the Jordan normal form of $J(E,E)$. The latter has two zero eigenvalues with eigenvector $\psi^1=(1,1)^T$. This vector will be the first column of the matrix of transformations $\Upsilon$. For the second column of $\Upsilon$ we choose the generalised eigenvector $\psi^2=(1,0)^T$. Thus,
$\Upsilon = \begin{pmatrix} 
1 & 1 \\ 
1 & 0 
\end{pmatrix}$ and 
\[ \begin{aligned}
\Upsilon^{-1}(J(E,E))\Upsilon &= \dfrac{S}{4}(H+M+1-A)(H+M+1+A)\begin{pmatrix}
0 & 1 \\ 
0 & 0 
\end{pmatrix} . 
\end{aligned}\]
Hence, we have the Bogdanov--Takens bifurcation, or bifurcation of codimension two, and the equilibrium point $(E,E)$ is a cusp point for $(Q,S)=(Q^{**},S^*(Q^{**}))$ such that $\Delta=0$ and $f(E)=S$ \cite{xiao2}, see Figure \ref{fig:05}.
\end{proof} 

\begin{figure}
\centering
\includegraphics[width=6.5cm]{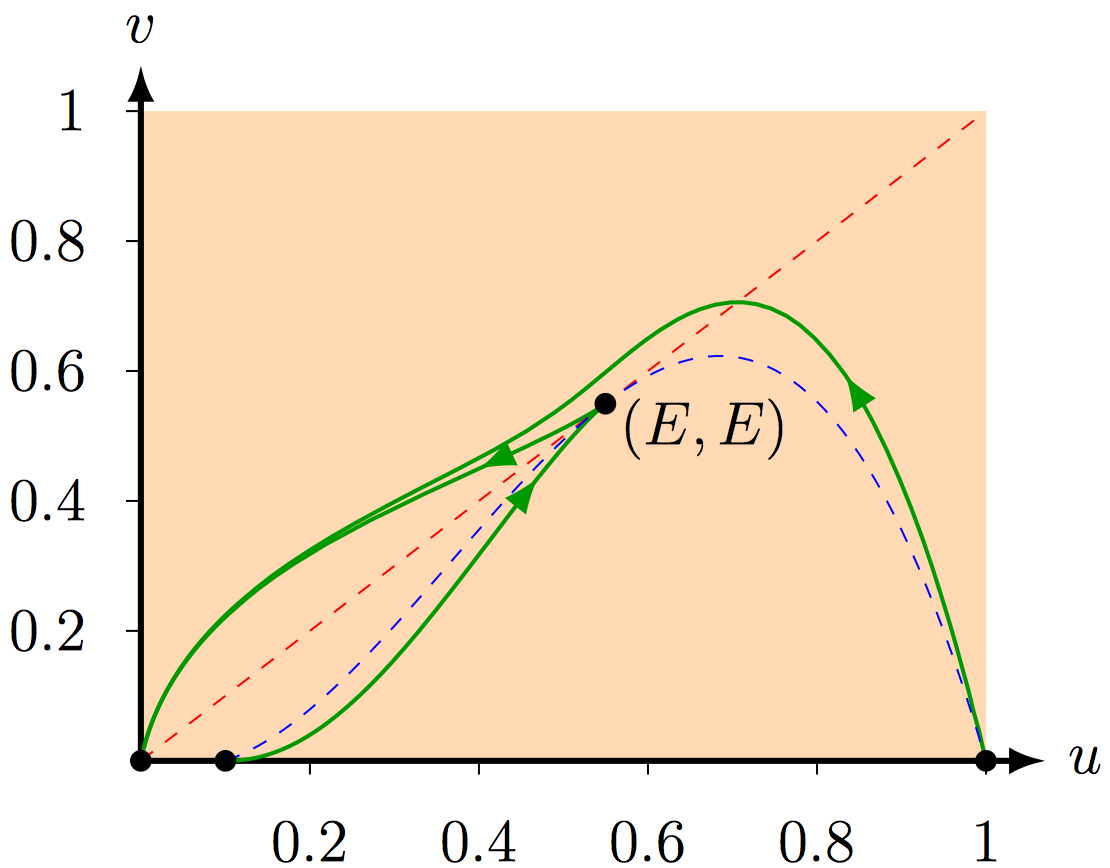}
\caption{For $M=0.1$, $A=0.0009$, $Q=0.20283145$, and $S=0.20249927$, such that $\Delta=0$ and $f(E)=S$, the point $(0,0)$ is an attractor and the equilibrium point $(E,E)$ is a cusp point. (See Figure \ref{fig:06} for the color conventions.)}
\label{fig:05}
\end{figure}

\begin{figure}
\centering
\includegraphics[width=12cm]{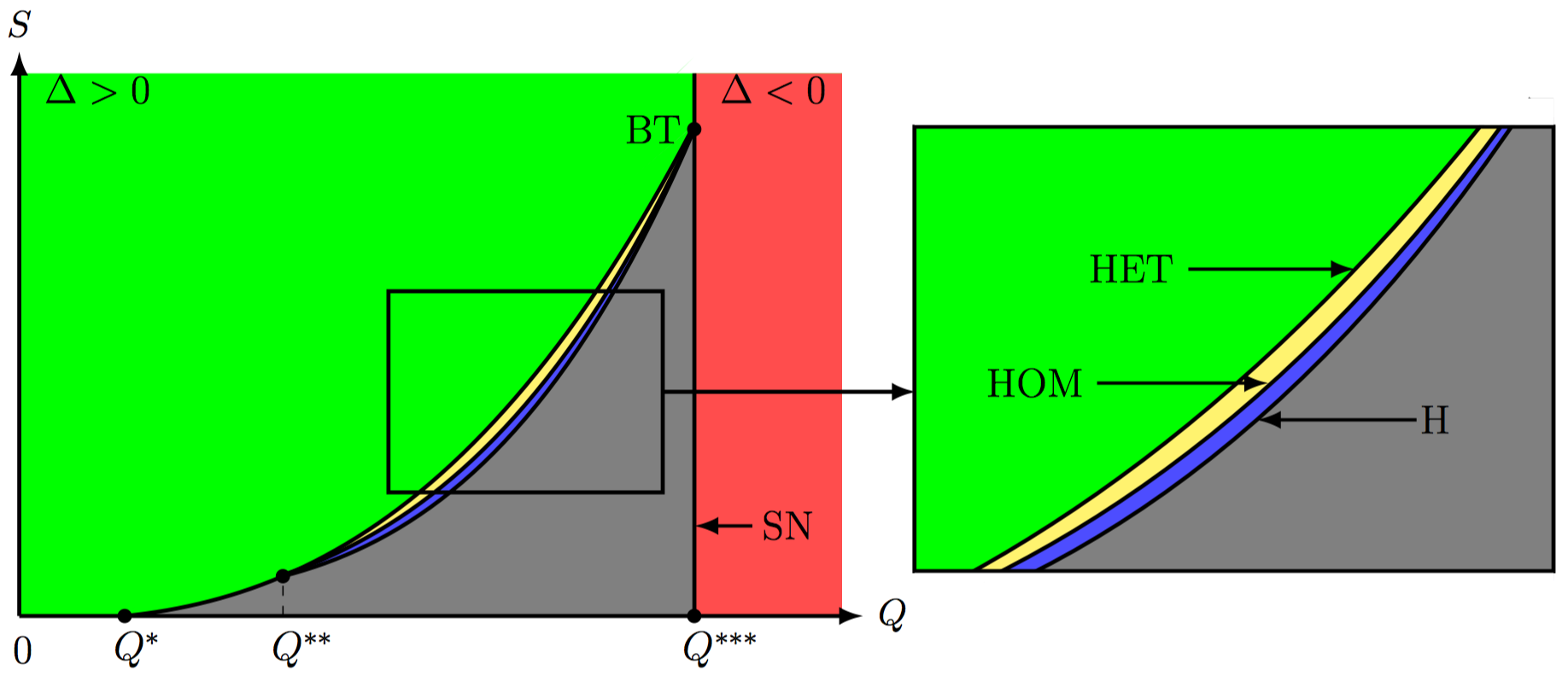}
\caption{The bifurcation diagram of system \eqref{hta2} for $A,M$ fixed and created with the numerical bifurcation package MATCONT \cite{matcont}. The curve H represents the Hopf curve at $S=S^*=f(u_2)$ where $P_2$ changes stability (Lemma~\ref{lemm4}), HOM represents the homoclinic curve of $P_1$ at $S=S^{**}$, HET represents the heteroclinic curve at $S=S^{***}$ connecting $P_1$ and $(1,0)$, SN represents the saddle-node curve from Lemma~\ref{lemm5} where $\Delta=0$, and $BT$ represents the Bogdanov--Takens bifurcation from Theorem~\ref{the4} where $\Delta=0$ and $S=F(E)$. Note that $P_2$ is always an attractor for $Q<Q^*$.}
\label{fig:03}
\end{figure}
The bifurcation curves obtained from Lemma~\ref{lemm4}, Lemma~\ref{lemm5}, and  Theorem \ref{the4} divide the $(Q,S)$-parameter-space into five parts, see Figure \ref{fig:03}. Modifying the parameter $Q$ -- while keeping the other two parameters $A,M$ fixed -- impacts the number of positive equilibrium points of system \eqref{hta2}. The modification of the parameter $S$ changes the stability of the positive equilibrium point $P_2$ of system \eqref{hta2}, while the other equilibrium points $(0,0)$, $(M,0)$, $(1,0)$ and $P_1$ do not change their behaviour. 
There are no positive equilibrium points in system \eqref{hta2} when the parameters $Q,S$ are located in the red area where $\Delta<0$ \eqref{delta}.
In this case, the origin is a global attractor, see Lemma~\ref{lemm2} and Figure \ref{fig:04}.  
For $Q=Q^{***}$, which is the saddle-node curve SN in Figure~\ref{fig:03}, the
equilibrium points $P_1$ and $P_2$ collapse since $\Delta=0$, see Lemma \ref{lemm5} and Figure \ref{fig:saddlenode}.
So, system \eqref{hta2} experiences a saddle-node bifurcation and a Bogdanov--Takens bifurcation (labeled BT in Figure~\ref{fig:03}) along this line, see Theorems \ref{the5} and \ref{the4}, and see also Figure \ref{fig:05}. When the parameters $Q,S$ are to the left of the line $Q=Q^{***}$, system \eqref{hta2} has two equilibrium points $P_1$ and $P_2$. The equilibrium point $P_1$ is always a saddle point, see Lemma~\ref{p1}, while $P_2$ can be unstable (grey area, see also Figure~\ref{fig:07}) or stable.
For $(Q,S)$ in the blue area the stable equilibrium point $P_2$ is surrounded by an unstable limit cycle, see also Figure~\ref{fig:06}. 
For $(Q,S)$ in the yellow area the basin of attraction of the stable equilibrium point $P_2$ is formed by $W^s(P_1)$ which connects $P_1$ with $(M,0)$, see also Figure~\ref{fig:08a}(a).
Finally, for $(Q,S)$ in the green area the basin of attraction of the stable equilibrium point $P_2$ is formed by $W^s(P_1)$ which connects to the boundary of $\Phi$, see also Figure~\ref{fig:08}.

\section{Conclusions}
In this manuscript, the Holling--Tanner predator-prey model with strong Allee effect and functional response Holling type II, i.e \eqref{hta1} with $m>0$, was studied. Using a diffeomorphism, see Theorem \ref{the1}, we analysed a topologically equivalent system \eqref{hta2}. This system has four system parameters which determine  the number and the stability of the equilibrium points. We showed that the equilibrium points $(1,0)$ and $P_1$ are always hyperbolic saddle points, $(M,0)$ is a hyperbolic repeller, and $(0,0)$ is a non-hyperbolic attractor, see Lemmas~\ref{eqax}--\ref{p1}. In contrast, the equilibrium point $P_2$ can be an attractor or a repeller, depending on the trace of its Jacobian matrix, see Lemma \ref{lemm4}. Furthermore, for some sets of parameters values the stable manifold of $P_1$ determines a separatrix curve which divides the basins of attraction of $(0,0)$ and $P_2$, see Figures \ref{fig:07}--\ref{fig:08}, while the separatrix curve is determined by an unstable limit cycle for other parameters values, see Figure \ref{fig:06}(a). 
The equilibrium points $P_1$ and $P_2$ collapse for $\Delta=0$ \eqref{delta} and system \eqref{hta2} experiences a saddle-node bifurcation, see Theorem \ref{the5}. Additionally, for $S=f(E)$ we obtain a cusp point (Bogdanov--Takens bifurcation) \cite{xiao2}, see Theorem \ref{the4}. We summarise the behavior for changing parameters $S$ and $Q$ in Figure \ref{summ}.

\begin{figure}
\centering
\includegraphics[width=12cm]{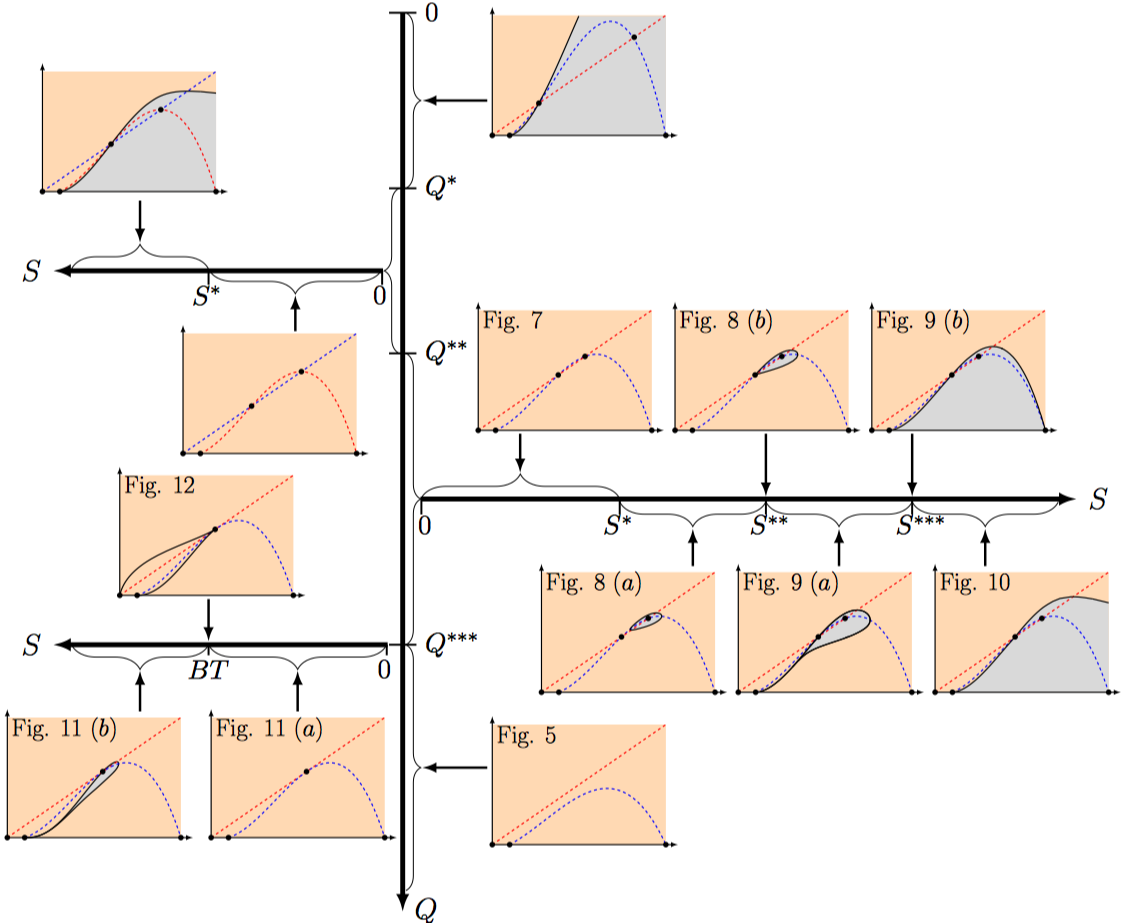}
  
 \caption{Summary of the basins of attraction of $P_2$ and $(0,0)$ for varying $S$ and $Q$ such that $\Delta>0$, $\Delta=0$ (at $Q=Q^{***}$) and $\Delta<0$. Note that the system parameters $Q$ and $S$ are such that $0<S^*<S^{**}<S^{***}$ and  $0<Q^*<Q^{**}<Q^{***}$.}
\label{summ}
\end{figure}

Since the function $\varphi$ is a diffeomorphism preserving the orientation of time, the dynamics of system \eqref{hta2} is topologically equivalent to system \eqref{hta1}, see Theorem \ref{the1} and Figure \ref{dif}. Therefore, we can conclude that for certain population sizes, there exists self-regulation in system \eqref{hta1}, that is, the species can coexist. 
However, system \eqref{hta1} is sensitive to disturbances of the parameters, see the changes of the basin of attraction of $P_2$ in Figure \ref{summ}. In addition, we showed that the self-regulation depends on the values of the parameters $S$ and $Q$. Since $S:=s/(rK)$ (see Theorem~\ref{the1}), this, for instance, implies that increasing the intrinsic growth rate of the predator $r$ -- or the carrying capacity $K$ -- decreases the area of coexistence (related to basins of attraction of $P_2$ in \eqref{hta2}), or, equivalent, decreasing the intrinsic growth rate of the prey $s$ decreases this area of the coexistence. Similar statements can be derived for the other system parameters of \eqref{hta1}.
From the basins of attractions it also follows that -- for a large range of parameter values -- coexistence is expected when the initial prey population is considerably higher than the initial predator population. However, when the proportion of predators is bigger than the proportion of preys both populations go to extinction.

Additionally, we showed that the strong Allee effect in the Holling--Tanner model \eqref{hta1} modified the dynamics of the original Holling--Tanner model \eqref{ht1}. Saez and Gonzalez-Olivares \cite{saez} showed that system \eqref{ht1} has always one positive equilibrium point which can be stable, or unstable surrounded by a stable limit cycle, or stable surrounded by two limit cycles. So, the population could coexist but could not extinct. The modified Holling--Tanner model \eqref{hta1} allows for this duality. Additionally, we showed that the Holling type II functional response presented in this manuscript changed the dynamics of the model with Type III response function \cite{pal,tintinago}, since the limit cycles presented in each study had different stability.

\bibliography{References.bib}
\end{document}